\documentclass[reqno,a4paper]{amsart}
\usepackage{amsmath,amsthm,amssymb}
\usepackage{xcolor}
\usepackage{fullpage}
\newtheorem{Theorem}{Theorem}[section]
\newtheorem{Definition}[Theorem]{Definition}
\newtheorem{Corollary}[Theorem]{Corollary}
\newtheorem{Lemma}[Theorem]{Lemma}
\newtheorem{Proposition}[Theorem]{Proposition}
\newtheorem{Remark}{Remark}[section]

\begin{document}
%[[[ titles
\title[Decay estimate for subcritical semilinear damped wave equations ]{Decay estimate for subcritical semilinear damped wave equations with slowly decreasing data}
\author[K.Fujiwara]{Kazumasa Fujiwara }
\address[K. Fujiwara]{
	Faculty of Advanced Science and Technology,
	Ryukoku University,
	1-5 Yokotani,Seta Oe-cho,Otsu,Shiga,
	520-2194, Japan
}

\email{fujiwara.kazumasa@math.ryukoku.ac.jp}

\author[V.Georgiev]{Vladimir Georgiev}
\address[V.Georgiev]{
Department of Mathematics,
University of Pisa,
Largo Bruno Pontecorvo 5,
I - 56127 Pisa, Italy}
\address{
Faculty of Science and Engineering, Waseda University,
3-4-1, Okubo, Shinjuku-ku, Tokyo 169-8555, Japan}
\address{
Institute of Mathematics and Informatics,  Bulgarian Academy of Sciences, Acad. Georgi Bonchev Str., Block 8, Sofia, 1113, Bulgaria
}
\email{georgiev@dm.unipi.it}

%]]]

%[[[ \begin{abstract}
\begin{abstract}
We study the decay properties of non-negative solutions
to the one-dimensional defocusing damped wave equation
in the Fujita subcritical case under a specific initial condition.
Specifically, we assume that the initial data
are positive, satisfy a condition ensuring the positiveness of solutions,
and exhibit polynomial decay at infinity.

To show the decay properties of the solution,
we construct suitable supersolutions composed of an explicit function
satisfying an ordinary differential inequality
and the solution of the linear damped wave equation.
Our estimates correspond to the optimal ones inferred
from the analysis of the heat equation.
\end{abstract}
%]]]

\maketitle

%[[[ \section{Introduction}
\section{Introduction}
%[[[ We consider the pointwise decay behavior
We consider the pointwise decay behavior
of solutions to the Cauchy problem for the following damped wave equations
under certain initial condition:
	\begin{align}
	\begin{cases}
	\partial_t^2 u + \partial_t u - \Delta u + u^p = 0,
	& t \in (0, \infty) ,\ x \in \mathbb R\\
	u(0,x) = u_0(x) & x \in \mathbb R,\\
	\partial_t u(0,x) = u_1(x) & x \in \mathbb R.
	\end{cases}
	\label{eq:DW}
	\end{align}
We assume that $p > 1$
and the initial data $u_0, u_1$ are sufficiently regular
and satisfy the conditions
	\begin{align}
	\inf_{x \in \mathbb R} u_0(x),
	\quad
	\inf_{x \in \mathbb R} (u_1(x) + \frac 1 2 u_0(x)) \geq 0.
	\label{eq:IC}
	\end{align}
%]]]

%[[[ We call $u$ a mild solution to \eqref{eq:DW}
We call $u$ a mild solution to \eqref{eq:DW}
if $u$ is non-negative and it satisfies the integral form of \eqref{eq:DW}
given by
	\begin{align}
	u(t,x)
	= S(t)(u_0+u_1) + \partial_t S(t) u_0
	- \int_0^t S(t-\tau) u^p(\tau) d\tau,
	\label{eq:DWIntegral}
	\end{align}
where $S(t)$ denotes the solution operator defined by
	\[
	S(t) f(x)
	= \frac{e^{-t/2}}{2}
	\int_{-t}^{t} I_0\bigg( \frac{\sqrt{t^2-y^2}}{2} \bigg) f(x-y) dy
	\]
for locally integrable function $f$.
Here, $I_0$ denotes the modified Bessel function of the first kind.
For details, see \cite[Chapter 6 Section 5 Part 7]{CH89}.
Namely, $S(t) f$ solves the Cauchy problem:
	\[
	\begin{cases}
	(\partial_t^2 + \partial_t - \Delta) S(t) f = 0,
	& t \in (0,\infty), \ x \in \mathbb R,\\
	S(0) f = 0,
	& x \in \mathbb R,\\
	\partial_t S(0) f = f,
	& x \in \mathbb R.
	\end{cases}
	\]
%]]]

%[[[ To investigate the existence of mild solutions to \eqref{eq:DW},
To investigate the existence of mild solutions to \eqref{eq:DW},
we review known results for the Cauchy problem of the damped wave equation:
	\begin{align}
	\begin{cases}
	\partial_t^2 u - \Delta u + \partial_t u = \lambda |u|^{p-1} u,
	& t \in (0, \infty) ,\ x \in \mathbb R,\\
	u(0,x) = u_0(x),
	& x \in \mathbb R,\\
	\partial_t u(0,x) = u_1(x),
	& x \in \mathbb R
	\end{cases}
	\label{eq:DWg}
	\end{align}
with $\lambda \in \mathbb R$ and $p > 1$.
A mild solution, allowing sign changes,
can be constructed by applying a standard contraction mapping principle
and the estimates for the solution operator given
by Marcati-Nishihara \cite{MN03}.
In particular, they showed the asymptotic equivalence
between $S(t)$ and the heat kernel $e^{t \Delta}$.
%]]]

%[[[ Under the initial condition \eqref{eq:IC},
The existence of mild-solutions to \eqref{eq:DW} is basically guaranteed
by the argument above and the a priori estimates of Li and Zhou \cite{LZ95}
under the initial condition \eqref{eq:IC}.
Indeed, if $u_0, u_1 \in C^\infty$,
then the following a priori estimate holds for any $x$ and $t$,
provided that $u$ is a mild solution of \eqref{eq:DW}:
	\begin{align}
	&e^{-t/2} \frac{u_0(x+t)+u_0(x-t)}{2}
	+ e^{-t/2} \int_{x-t}^{x+t} u_1(y) + \frac 1 2 u_0(y) dy
	\nonumber \\
	&\leq u(t,x)
	\leq S(t)(u_0+u_1)(x) + \partial_t S(t) u_0(x).
	\label{eq:apriori}
	\end{align}
Combining the a priori estimate \eqref{eq:apriori}
and contraction mapping argument discussed above,
mild solutions to \eqref{eq:DW} can be constructed.
In the case where $(u_0, u_1)
\in W^{1,1} \cap W^{1,\infty} (\mathbb R)
\times L^{1} \cap L^{\infty} (\mathbb R)$,
\eqref{eq:apriori} has also been established in \cite{FG04o}.
See also Proposition \ref{Proposition:Comparison} in Appendix.
%]]]

%[[[ On the other hand,
On the other hand,
it is known that \eqref{eq:apriori}
is not sufficient to describe the decay rate of the solution $u$.
Indeed, by rewriting $-u$ as $\widetilde u$,
we have
	\[
	\begin{cases}
	\partial_t^2 \widetilde u
	+ \Delta \widetilde u
	- \Delta \widetilde u = |\widetilde u|^p,\\
	\widetilde u(0,x) = -u_0(x),\\
	\partial_t \widetilde u(0,x) = -u_1(x).
	\end{cases}
	\]
By integrating the equation above formally with respect to $x$ and $t$,
	\[
	\int_{\mathbb R} \widetilde u(t,x) + \partial_t \widetilde u(t,x) dx
	= - \int_{\mathbb R} u_0(x) + u_1(x) dx
	+ \| \widetilde u \|_{L^p(0,t ; L^p( \mathbb R))}^p.
	\]
Then $u$ belongs to $L^p (0,\infty ; L^p(\mathbb R))$
as long as it exists globally.
Indeed, if $u \not\in L^p (0,\infty ; L^p(\mathbb R))$,
then there exists $T_0$ such that
	\[
	\int_{\mathbb R} \widetilde u(T_0,x) + \partial_t \widetilde u(T_0,x) dx
	> 0.
	\]
However,
Li and Zhou \cite{LZ95} showed that
there exists $T_1 > T_0$ such that
	\[
	\liminf_{t \to T_1} \inf_{|y| \leq \sqrt t} \sqrt t |\widetilde u(t,y)|
	= \infty,
	\]
which contradicts the global existence of $u$.
Since $S(t)$ is asymptotically equivalent to $e^{t \Delta}$,
if $S(t) f \in L^p(0,\infty;L^p(\mathbb R))$ for any $f \in L^1$,
then the estimate $p > 3$ is necessary.
Therefore, when $p \leq 3$(in the Fujita critical and subcritical cases),
solutions $u$ must decay faster than free solutions for some initial data.
Namely, the asymptotic behavior of the solution is primarily governed
by nonlinear effects rather than perturbative effects.
This nonlinear effect has attracted attention.
%]]]

%[[[ In order to study the nonlinear effect in the subcritical case,
In order to study the nonlinear effect of \eqref{eq:DW} in the subcritical case,
the decay behavior of solutions to the corresponding semilinear heat equation
has been considered as an indicator.
Let us recall the known results for the Cauchy problem
of the following semilinear heat equation:
	\begin{align}
	\begin{cases}
	\partial_t v - \Delta v + v^p = 0,\\
	v(0,x) = v_0(x),
	\end{cases}
	\label{eq:Heat}
	\end{align}
where $v_0$ and $v$ are non-negative functions.
Roughly speaking,
if $v_0$ decays sufficiently fast,
then the behavior of the solution $v$ is approximated
as $t \to \infty$
as follows
	\begin{align}
	v(t,x)
	\sim \begin{cases}
	t^{-1/(p-1)} h(|x|/t^{1/2}) & \mathrm{if} \quad 1< p < 3,\\
	\theta (\log t)^{-1/2} t^{-1/2} e^{-x^2/4t} & \mathrm{if} \quad p=3,\\
	\theta t^{-1/2} e^{-x^2/4t} & \mathrm{if} \quad p > 3.
	\end{cases}
	\label{eq:HeatAsymptotic}
	\end{align}
Here $\theta$ is a constant determined by the initial data
and $h : \lbrack 0,\infty) \to \lbrack 0,\infty)$ solves the ODE:
	\begin{align}
	\begin{cases}
	- \ddot h - \frac r 2 \dot h + h^p = \frac{1}{p-1} h,
	&r \in (0,\infty),\\
	\dot h(0)=0,
	\end{cases}
	\end{align}
and $t^{-1/2} e^{-x^2/4t}$ is the fundamental solution of the heat equation.
Therefore, in the subcritical case,
solutions $v$ is asymptotically equivalent to self-similar function.
For details, see \cite{EK88,EKM95,GKS85}.
Moreover,
Herraiz \cite[Theorem 1]{H99} showed that
the solution $v$ behaves like
	\begin{align}
	v(t,x)
	\sim
	\begin{cases}
	t^{-1/(p-1)}
	\Big(
		(p-1)^{p-1} - C t^{1-\frac{\rho(p-1)}{2}} h(|x|/t^{1/2})
	\Big)
	& \mathrm{if} \quad |x| \lesssim t^{1/2},\\
	\Big(
		(p-1)t + A^{1-p}|x|^{\rho(p-1)}
	\Big)^{-1/(p-1)}
	& \mathrm{if} \quad t^{1/2} \lesssim |x| \lesssim t^{1/\rho(p-1)},\\
	A |x|^{-\rho}
	& \mathrm{if} \quad t^{1/\rho(p-1)} \lesssim |x|.
	\end{cases}
	\label{eq:HerraizAsymptotic}
	\end{align}
when $1 < p < 3$ and $v_0$ decays slower than $h$ above,
namely
	\[
	v_0(x) \sim A |x|^{-\rho}
	\]
with $A> 0$ and $0 < \rho < 2/(p-1)$ as $|x| \to \infty$.
We note that $g(t,\varepsilon) = ((p-1)t+\varepsilon^{1-p})^{-1/(p-1)}$ solves the ODE:
	\[
	\dot g + g^p = 0
	\]
and
	\[
	g(t,A|x|^{-\rho})
	=
	\Big(
		(p-1)t + A^{1-p}|x|^{\rho(p-1)}
	\Big)^{-1/(p-1)}.
	\]
Therefore, in the subcritical case,
the behavior of solutions to \eqref{eq:Heat} may be summarized
from the viewpoint of the decay rate
of $L^q$-norm of the solutions as follows:
if $v_0 = \varepsilon \langle x \rangle^{-\rho}$,
where $\varepsilon$ is small enough and $\langle x \rangle = (1+|x|^2)^{1/2}$,
then
	\begin{align}
	\| v(t) \|_{L^q(\mathbb R)}
	\sim
	\begin{cases}
	t^{\frac{1}{2q}-\frac{1}{p-1}} & \mathrm{if} \quad \rho > \frac{2}{p-1},\\
	t^{\frac{1}{q\rho(p-1)}-\frac{1}{p-1}} & \mathrm{if} \quad 1 < \rho < \frac{2}{p-1}.
	\end{cases}
	\label{eq:ideal}
	\end{align}
for $q \in [1,\infty]$,
ensuring that $v \in L^p(0,\infty; L^p(\mathbb R))$.
%]]]

%[[[ Comparing to \eqref{eq:Heat},
Comparing to \eqref{eq:Heat},
the behavior of solutions to \eqref{eq:DW} is much less understood.
Hayashi, Kaikina, and Naumkin \cite{HKN04} showed that an asymptotic equivalence
similar to \eqref{eq:HeatAsymptotic} holds for the solution $u$ to \eqref{eq:DW},
when $(u_0,u_1)
\in (H^2 \cap W^{2,1} \cap \langle x \rangle^{a} L^1) \times (H^1 \cap W^{1,1} \cap \langle x \rangle^{a} L^1)$
with $a > 0$,
where $h$ is replaced appropriately.
However, in the subcritical case,
the asymptotic equivalence is shown when $3-\varepsilon^3 < p < 3$,
where $\varepsilon > 0$ represents the size of initial data and is small enough.
Later, Nishihara and Zhao \cite{NZ06} and Ikehata, Nishihara, and Zhao \cite{INZ06}
showed that in the subcritical case,
solutions $u$ enjoy a weighted energy estimate guaranteeing
the estimate
	\[
	\| u(t) \|_{L^q(\mathbb R)}
	\lesssim t^{\tfrac{1}{2q} - \tfrac{1}{p-1}}
	\]
for any $q \in [1,\infty]$,
when the following initial condition is satisfied:
	\[
	(1 + |x|)^{\frac{2}{p - 1} - \frac{1 - \delta}{2}}
	\bigl(u_0, \nabla u_0, u_1, |u_0|^{\tfrac{p+1}{2}}\bigr)
	\;\in\; L^2
	\]
with some small positive $\delta$.
We note that $u_0, u_1 \sim \langle x \rangle^{-\rho}$
with $\rho > \frac{2}{p-1}$,
satisfies the condition above.
We also note that Wakasugi \cite{W23}
generalized the results of \cite{INZ06}.
However, there is still a gap
between the known decay rate
for the solutions to \eqref{eq:DW} and \eqref{eq:HerraizAsymptotic}.
We remark that $L^\infty$ control corresponds to \eqref{eq:HerraizAsymptotic}
is known to hold under an initial regularity condition and \eqref{eq:IC}
in \cite{FG24}.
%]]]

%[[[ The aim of this manuscript is to show a pointwise decay estimate
The aim of this manuscript is to show a pointwise decay estimate
of solutions to \eqref{eq:DW}
comparable to \eqref{eq:HerraizAsymptotic}
under a certain initial condition.
The main result is as follows:
\begin{Theorem}
\label{Theorem:Main}
Let $(u_0,u_1)
\in
W^{1,1} \cap W^{1,\infty} (\mathbb R)
\times
L^{1} \cap L^{\infty} (\mathbb R).
$
Let $\phi \in W^{1,1} \cap W^{1,\infty} (\mathbb R; (0,\infty))$.
We assume that there exists a constant $C$ such that
the estimates
	\begin{align}
	\phi(x) &\leq C \inf_{y \in (x-1,x+1)} \phi(y),
	\label{eq:IC2}\\
	\phi(x) &\leq C \inf_{|y|<2|x|} \phi(y),
	\label{eq:IC3}\\
	|\dot \phi(x)| &\leq C \phi(x)
	\label{eq:IC4}
	\end{align}
hold for almost every $x \in \mathbb R$.
We further assume that for any $\delta >0$,
there exists a positive constant $C_\delta$ such that
	\begin{align}
	e^{-\delta |x|} \leq C_\delta \phi(x)
	\label{eq:IC5}
	\end{align}
holds for almost every $x \in \mathbb R$.
Then there exist a positive constant $C$,
$T_0 \geq 0$ large enough,
and $\varepsilon_0 > 0$ small enough
such that if $u_0$ and $u_1$ satisfy
	\[
	0 \leq u_0(x), u_1(x) + \frac 1 2 u_0(x)
	\leq C \frac{\varepsilon_0}{T_0^{p-1/2}} \phi(x)
	\]
for almost every $x \in \mathbb R$,
then solutions $u$ to \eqref{eq:DW}
satisfy
	\begin{align}
	0
	\leq u(t,x)
	\leq C \bigg(
		\frac{u_L(t+T_0,x)^{p-1}}{(t+T_0) u_L(t+T_0,x)^{p-1} + 1}
	\bigg)^{\frac{1}{p-1}},
	\label{eq:Main}
	\end{align}
for almost every $x \in \mathbb R$ and $t \geq 0$,
where $u_L = \varepsilon_0 S(t) \phi + \varepsilon_0 \partial_t S(t) \phi$.
\end{Theorem}
%]]]

%[[[ We note that
We note that
$\phi(x) = \langle x \rangle^{-\rho}$ with $\rho > 1$
satisfies the condition of Theorem \ref{Theorem:Main}.
Then we have the following corollary:
\begin{Corollary}
\label{Corollary:Main}
Let $\varepsilon > 0$ be small enough and $1 < \rho < \frac{2}{p-1}$.
Let $0 \leq u_0, u_1 + u_0 / 2 \leq \varepsilon \langle x \rangle^{-\rho}$.
Then, $u$ enjoys the estimate
	\[
	\| u(t) \|_{L^q} \lesssim t^{\tfrac{1}{q\rho(p-1)} - \tfrac{1}{p-1}}
	\]
for $q \in [1,\infty]$.
Especially $u \in L^p(0,\infty; L^p(\mathbb R))$.
\end{Corollary}
Therefore,
we conclude that,
under the initial condition of Corollary \ref{Corollary:Main},
the solution $u$ to \eqref{eq:DW} decays
as fast as the solution to \eqref{eq:Heat} with the same initial condition.
%]]]

%[[[ We note that Theorem \ref{Theorem:Main} holds for any $\rho > 1$
We note that Theorem \ref{Theorem:Main} holds for any $\rho > 1$
but the sharp upperbound \eqref{eq:ideal} can be obtained
by Theorem \ref{Theorem:Main} only when $\rho < \frac{2}{p-1}$.
Indeed,
a logarithmic loss occurs
when one estimate the $L^q$ norm
of the RHS of the second estimate of \eqref{eq:Main}
with $\phi(x) = \langle x \rangle^{-\rho}$ with $\rho \geq \frac{2}{p-1}$.
For details,
see Proposition \ref{Proposition:HeatSuperSolution2} in Appendix.
This logarithmic loss also implies that
solutions to \eqref{eq:DW} cannot be estimated from below
by a function taking the form of the RHS of the second estimate of \eqref{eq:Main}
at least when $\rho \geq \frac{2}{p-1}$.
%]]]

%[[[ Our approach is to show Theorem \ref{Theorem:Main}
Our approach showing Theorem \ref{Theorem:Main}
is to compare solution $u$ with supersolutions.
For the definition of supersolution,
see Section \ref{Section:Proof}.
We assume the initial condition \eqref{eq:IC}
to apply a comparison argument.
For details, see Appendix \ref{Proposition:Comparison}.
We construct supersolutions by combining ODE supersolutions and free solutions.
To illustrate the idea,
we reconsider \eqref{eq:Heat}.
The ODE corresponding to \eqref{eq:Heat} is
	\begin{align}
	\begin{cases}
	\dot g + g^p = 0,\\
	g(0) = f \geq 0.
	\end{cases}
	\label{eq:HODE}
	\end{align}
Here, $g(t)=G(t,f) = ((p-1)t+f^{1-p})^{-1/(p-1)}$.
Then, we claim that $G^\ast(t,x) = G(t,e^{t \Delta} \phi(x))$ is a supersolution
to \eqref{eq:Heat}
when $\phi$ is a non-negative and non-trivial function.
Indeed,
we compute that
	\[
	\partial_f G(t,f) = G^p f^{-p},
	\quad
	\partial_f^2 G(t,f) = p G^{p} f^{-2p} ( G^{p-1} - f^{p-1}),
	\]
and therefore $\partial_f^2 G < 0$ if $f > 0$.
Then we further compute
	\begin{align*}
	&\partial_t G^\ast - \Delta G^\ast + (G^\ast)^p\\
	&= (\partial_t G + G^p)^\ast
	+ (\partial_f G)^\ast (\partial_t e^{t \Delta} \phi - \Delta e^{t \Delta} \phi)
	- (\partial_f^2 G)^\ast (\partial_x \phi)^2
	> 0.
	\end{align*}
This implies that the pointwise estimate
$v(t,x) \leq G^\ast(t,x)$ holds.
We remark that
if $\phi(x) = \langle x \rangle^{-\rho}$ with $1 < \rho < \frac{2}{p-1}$,
then the estimate
	\[
	\| G^\ast (t) \|_{L^q} \lesssim t^{\frac{1}{q \rho(p-1)} - \frac{1}{p-1}}
	\]
holds for any $q \in [1,\infty]$.
Therefore,
the upper estimate of \eqref{eq:ideal}
is obtained by this comparison argument for small $\rho$.
For details, see Proposition \ref{Proposition:HeatSuperSolution}
in Appendix.
%]]]

%[[[ We note that,
We note that,
since
$\ddot g(t)
= p g^{2p-1}$,
$g$ is a supersolution to the ODE:
	\begin{align}
	\ddot f + \dot f + f^p = 0,
	\label{eq:ODEDW}
	\end{align}
which corresponds to \eqref{eq:DW}.
We remark that a comparison argument for \eqref{eq:ODEDW}
holds under an initial condition
	\[
	f(0) \geq 0,
	\quad
	\dot f(0) + \frac 1 2 f(0) \geq 0.
	\]
One can check this fact by a similar argument
to Proposition \ref{Proposition:Comparison}.
On the other hand,
$\widetilde G(t) = G(t,u_L(t))$ may not be a supersolution to \eqref{eq:DW},
where $u_L(t) = S(t) \phi + \partial_t S(t) \phi$ solving
	\begin{align}
	\begin{cases}
	\partial_t^2 u_L + \partial_t u_L - \Delta u_L = 0,\\
	u_L(0) = \phi,\\
	\partial_t u_L(0) = 0.
	\end{cases}
	\label{eq:ODEDW60}
	\end{align}
Indeed, by putting $\widetilde G^\ast(t,x) = \widetilde G(t,u_L(t,x))$,
we compute
	\begin{align*}
	&\partial_t^2 \widetilde G^\ast
	+ \partial_t \widetilde G^\ast
	- \Delta \widetilde G^\ast + (\widetilde G^\ast)^p\\
	&= (
		\partial_t^2 \widetilde G + \partial_t \widetilde G + \widetilde G^p
	)^\ast
	+ (\partial_f \widetilde G)^\ast
	(\partial_t^2 + \partial_t - \Delta) u_L
	+ 2 (\partial_t \partial_f \widetilde G)^\ast \partial_t u_L
	+ (\partial_f^2 \widetilde G)^\ast
	((\partial_t u_L)^2 - (\partial_x u_L)^2)\\
	&= p (\widetilde G^{2p-1})^\ast
	- 2 p (\widetilde G^{2p-1})^\ast u_L^{-p} \partial_t u_L
	+ p \Big( (\widetilde G/u_L)^{2p-2} - (\widetilde G/u_L)^{p-1} \Big)^\ast
	(\widetilde G)^\ast u_L^{-2}
	\Big((\partial_t u_L)^2 - (\partial_x u_L)^2\Big).
	\end{align*}
We note that $\widetilde G \leq u_L$ holds
because $\widetilde G(t,u_L(t)) \leq \widetilde G(0,u_L(t)) = u_L(t)$.
Then in the above computation,
it is unclear whether
$\widetilde G$ is a supersolution to \eqref{eq:DW}.
In this manuscript,
we overcome this difficulty
by using a modified supersolution to \eqref{eq:ODEDW60}.
%]]]

%[[[ This manuscript is organized as follows:
This manuscript is organized as follows:
In section \ref{Section:ODE},
we discuss a supersolution to the ODE related with \eqref{eq:DW}.
In section \ref{Section:Linear},
we show some estimates related with linear solutions to \eqref{eq:DW}.
In section \ref{Section:Proof},
we prove Theorem \ref{Theorem:Main} and Corollary \ref{Corollary:Main}.
%]]]
%]]]

%[[[ \section{ODE related with DW}
\section{ODE related with \eqref{eq:DW}}
\label{Section:ODE}
In this section,
we consider a supersolution to \eqref{eq:ODEDW}.
The construction of supersolutions
is the basis to show the decaying estimate Theorem \ref{Theorem:Main}.

%[[[ Proposition: ODE supersolution
\begin{Proposition}
\label{Proposition:ODESuperSolution}
Let $p > 1$ and $\varepsilon > 0$ small.
Let $\alpha \in (0,1]$.
Let
	\begin{align}
	\label{eq.defwg1}
	g(t, \varepsilon)
	= \bigg(
	\frac{p (p+1) \varepsilon^{p-1}}{(p-1)^2 \varepsilon^{p-1} t + p^2+p}
	\bigg)^{1/(p-1)},
	\quad
	w(t)
	= \bigg( \frac 1 2 + \frac{1}{(2^{1/\alpha} + t)^{\alpha}} \bigg)^{-1/(p-1)}.
	\end{align}
Then $H(t,\varepsilon) = w(t) g(t,\varepsilon)$ enjoys the following properties
with a positive constant $C_{p,\alpha}$ depending only on $p$ and $\alpha$:
\begin{equation}\label{eq.odest6}
	\left\{\begin{aligned}
	&\ddot H + \dot H + H^p
	\geq \frac{C_{p,\alpha}}{(1+t)^{\alpha+1}} H
	+ C_{p,\alpha} g^{p-1} H,
	&t > 0\\
	&\dot H + \frac 1 2 H \geq 0,
	&t \geq 0\\
	& H(0) = \varepsilon,\\
	&\dot H(0)
	= \frac{\alpha}{(p-1) 2^{(\alpha+1)/\alpha}} \varepsilon - \frac{p-1}{p(p+1)} \varepsilon^{p}.
	\end{aligned}\right.
\end{equation}
\end{Proposition}

\begin{proof}
For simplicity, we abbreviate $g(t,\varepsilon)$ as $g(t)$.
We compute
	\[
	g(0) = \varepsilon, \quad
	\dot g(t) = - \frac{p-1}{p(p+1)} g(t)^p, \quad
	\quad \ddot g(t) = \frac{(p-1)^2}{p(p+1)^2} g(t)^{2p-1} > 0
	\]
and
	\begin{align*}
	\dot w(t)
	&= - \frac{1}{(p-1)} w(t)^p
	\frac{d}{dt} \bigg( \frac 1 2 + \frac{1}{(2^{1/\alpha} + t)^{\alpha}} \bigg)\\
	&= \frac{\alpha}{(p-1)} \frac{1}{(2^{1/\alpha}+t)^{\alpha+1}} w^p(t).
	\end{align*}
These identities above imply that
$H= w g$ satisfies
	\[
	H(0) = \varepsilon, \quad
	\dot H(0) = g(0) \dot w(0) + w(0) \dot g(0)
	= \frac{\alpha}{(p-1) 2^{(\alpha+1)/\alpha}} \varepsilon - \frac{p-1}{p(p+1)} \varepsilon^{p}.
	\]
We also estimate
	\begin{align}
	\ddot w(t) + \dot w(t)
	&= \bigg( \frac{p \alpha}{(p-1)} \frac{1}{(2^{1/\alpha}+t)^{\alpha+1}} w^{p-1}(t)
	+ \frac{2^{1/\alpha} - \alpha-1 + t}{2^{1/\alpha}+t} \bigg) \dot w(t)
	\geq 0.
	\label{eq:ddw+dw}
	\end{align}
Here we have used the fact that
$2^{1/\alpha} - \alpha-1$
is strictly decreasing for positive $\alpha$
and is $0$ at $\alpha = 1$.
We further estimate
	\[
	\dot w(t)
	\leq \frac{\alpha}{(p-1)2^{1/\alpha}} w(t)
	\]
and
	\begin{align*}
	w^{p-1}(t) - \frac{p-1}{p+1}
	&\geq  1 - \frac{p-1}{p+1} \geq \frac{2}{p+1}.
	\end{align*}
Then the estimates above imply that
	\begin{align*}
	\ddot H + \dot H + H^p
	&= w \ddot g + (\ddot w + \dot w)g + (2 \dot w + w) \dot g + w^p g^p\\
	&\geq \bigg( w^p - \frac{p-1}{p(p+1)} w - \frac{2(p-1)}{p(p+1)} \dot w \bigg) g^p
	+ (\ddot w + \dot w)g\\
	&\geq \bigg( \frac{2}{p+1} - \frac{2\alpha}{p(p+1)2^{1/\alpha}} \bigg) w g^p
	+ \frac{C}{(2^{(\alpha+1)/\alpha}+t)^{\alpha+1}} H.
	\end{align*}
Moreover, we have
	\[
	\dot H + \frac 1 2 H
	= \dot w g + \frac 1 2 w g + w \dot g\\
	\geq \bigg( \frac 1 2 - \frac{p-1}{p(p+1)} g^{p-1} \bigg) g
	\geq 0,
	\]
where we have used the assumption that
$\varepsilon$ is small and
	\[
	\max_{1 \leq p \leq 3} \frac{p-1}{p(p+1)}
	=\frac{p-1}{p(p+1)} \bigg|_{p=1+\sqrt{2}}
	\sim 0.17 < \frac 1 2.
	\]
\end{proof}
%]]]

%[[[ \begin{Remark}
\begin{Remark}
%[[[ Summarizing the relations used in the proof of the previous Lemma,
Summarizing the relations used in the proof of the previous Lemma,
we can write
	\begin{equation}\label{eq.rem7}
	\left\{   \begin{aligned}
	&\dot w(t)
	= \frac{\alpha}{(p-1)} \frac{1}{(2^{1/\alpha}+t)^{\alpha+1}} w^p(t),\\
	& \ddot w(t)
	= \bigg( \frac{p \alpha}{(p-1)} \frac{1}{(2^{1/\alpha}+t)^{\alpha+1}} w^{p-1}(t)
	 +\frac{  \alpha+1}{2^{1/\alpha}+t} \bigg) \dot w(t).
	\end{aligned}\right.
	\end{equation}
We have also the following relations for the function $g(t,\varepsilon)$
defined in \eqref{eq.defwg1};
	\begin{equation}\label{eq.relcr7}
	g(t, \varepsilon)^{p-1}
	= \frac{p (p+1) \varepsilon^{p-1}}{(p-1)^2 \varepsilon^{p-1} t + p^2+p},
	\end{equation}
so we have the following partial derivatives up to order 2 in $t$
\begin{equation}\label{eq.rem34}
 \left\{   \begin{aligned}
	&
	\partial_t g
	= - \frac{p-1}{p(p+1)} g^p,\\
	& \partial_t^2 g= \frac{(p-1)^2}{p(p+1)^2} g^{2p-1} .
	\end{aligned}\right.
\end{equation}
Similarly, from \eqref{eq.relcr7} for the partial derivatives in $\varepsilon$ we have
\begin{equation}\label{eq.pade1}
  \left\{  \begin{aligned}
		&
	\partial_\varepsilon g
	= g^{p} \varepsilon^{-p},\\
	& \partial_\varepsilon^2 g = (-p\varepsilon^{-1} +p g^{p-1}) \partial_\varepsilon g =
	 (-p\varepsilon^{-1} +p g^{p-1}) g^p \varepsilon^{-p}.
	\end{aligned} \right.
\end{equation}
%]]]

%[[[ Finally, differentiating  the first relation in \eqref{eq.pade1}
Finally, differentiating  the first relation in \eqref{eq.pade1} in $t$
and using and the first relation in \eqref{eq.rem34},
we conclude that  the mixed derivative is
	\[
	\partial_t\partial_\varepsilon g = - \frac{p-1}{p+1} g^{2p-1}  \varepsilon^{-p}.
	\]
%]]]
\end{Remark}
%]]]

%[[[ The above Remark leads to the following.
The above Remark leads to the following.
\begin{Corollary} \label{cor1}
For $\varepsilon > 0$ and $t>0$,
we have the relations
	\begin{equation}\label{eq.rem61}
 \left\{   \begin{aligned}
	&
	\partial_t H
	= \left( \frac{\alpha w^{p-1}}{(p-1)(2^{1/\alpha}+t)^{\alpha+1}} - \frac{p-1}{p(p+1)} g^{p-1} \right) H,\\
	& \partial_\varepsilon	H= H g^{p-1} \varepsilon^{-p} ,\\
	& \partial_t \partial_\varepsilon H = \left(\frac{\alpha w^{p-1}g^{p-1}}{\varepsilon^{p}(p-1)(2^{1/\alpha}+t)^{\alpha+1}} - \frac{(p-1)g^{2p-2}}{(p+1) \varepsilon^p} \right) H ,\\
	& \partial_\varepsilon^2 H = \left(pg^{2p-2}\varepsilon^{-2p} - p g^{p-1} \varepsilon^{-p-1} \right) H \leq 0.
	\end{aligned}\right.
\end{equation}

\end{Corollary}
\begin{proof}
The last estimate holds because $g(t,\varepsilon) \leq \varepsilon$.
\end{proof}
%]]]
%]]]

%[[[  \section{Linear Solutions to Damped Wave Equation}
\section{Linear Solutions to Damped Wave Equation}
\label{Section:Linear}
%[[[ In this section,
In this section,
we collect the estimate for kernel parts of $S(t)$ and its derivatives.
For simplicity, we denote
	\[
	\omega = \sqrt{t^2-y^2}.
	\]
%]]]

%[[[ We discuss a pointwise estimate of $I_0$
We discuss a pointwise lower bound of $I_0$ by using the representation
	\begin{align}
	I_{0} (x)
	= \frac{1}{\pi} \int_{- x}^{x} \frac{e^{\eta}}{\sqrt{x^{2} - \eta^{2}}} d \eta.
	\label{eq:0-mBessel}
	\end{align}
For the detail of \eqref{eq:0-mBessel}, see \cite[8.431]{GR07}.
\begin{Lemma}
\label{Lemma:0-mBessel}
	\begin{align*}
	I_{0} (x) \geq \begin{cases}
		1 &\quad \text{if} \quad 0< x \leq 1, \\
		\frac{5}{6 \pi} \frac{e^{x}}{\sqrt{x}} &\quad \text{if} \quad x \geq 1.
	\end{cases}
	\end{align*}
\end{Lemma}

\begin{proof}
	Let $x >0$.
	Changing variables in \eqref{eq:0-mBessel} as $\eta = x \cos \theta$, we have
	\begin{align*}
		I_{0} \left( x \right)
		= \frac{1}{\pi} \int_{0}^{\pi} e^{x \cos \theta} d \theta
		= \frac{2}{\pi} \int_{0}^{\frac{\pi}{2}} \cosh \left( x \cos \theta \right) d \theta.
	\end{align*}
	If $x \leq 1$, then we get
	\begin{align*}
		I_{0} \left( x \right) &\geq \frac{2}{\pi} \int_{0}^{\frac{\pi}{2}} 1d \theta =1.
	\end{align*}
	We next consider the case where $x \geq 1$.
	We note that
	\begin{align*}
		1- \frac{\theta^{2}}{2} \leq \cos \theta \leq 1- \frac{\theta^{2}}{\pi}
	\end{align*}
	for any $\theta \in \left[ 0, \pi /2 \right]$.
	By using the above inequalities, we obtain
	\begin{align*}
		I_{0} \left( x \right)
		&\geq \frac{1}{\pi} \int_{0}^{\frac{\pi}{2}} e^{x \cos \theta} d \theta \\
		&\geq \frac{1}{\pi} e^{x} \int_{0}^{\frac{1}{\sqrt{x}}} e^{- \frac{x \theta^{2}}{2}} d \theta \\
		&= \frac{\sqrt{2}}{\pi} \frac{e^{x}}{\sqrt{x}} \int_{0}^{\frac{1}{\sqrt{2}}} e^{- \rho^{2}} d \rho \\
		&\geq \frac{\sqrt{2}}{\pi} \frac{e^{x}}{\sqrt{x}} \int_{0}^{\frac{1}{\sqrt{2}}} \left( 1- \rho^{2} \right) d \rho \\
		&= \frac{5}{6 \pi} \frac{e^{x}}{\sqrt{x}}.
	\end{align*}
	This completes the proof of Lemma \ref{Lemma:0-mBessel}.
\end{proof}
%]]]

%[[[ Next, we show pointwise upper bounds
Next, we show pointwise upper bounds
of $\partial_t S(t)$ and $\partial_t^2 S(t)$.
We note $\partial_t S(t)$ and $\partial_{t}^2 S(t)$ are represented
as follows:
	\begin{align}
	\partial_t S(t) f(x)
	&= e^{-t/2} \frac{f(x+t)+f(x-t)}{2}
	+ \widetilde{\mathcal K}_1(t) \ast f(x),
	\label{eq:St}\\
	\partial_t^2 S(t) f(x)
	&= e^{-t/2} \frac{f'(x+t)-f'(x-t)}{2}
	\nonumber\\
	&+ e^{-t/2} \bigg( \frac{t}{16} - \frac 1 2 \bigg)
	(f(x+t)+f(x-t))
	+ \widetilde{\mathcal K}_2(t) \ast f(x),
	\label{eq:Stt}
	\end{align}
where
	\begin{align*}
	\mathcal K_1(t,y)&= \frac 1 4 e^{-t/2}
	\bigg( \frac{t}{\omega} I_1(\frac \omega 2)
	- I_0(\frac \omega 2) \bigg),\\
	\mathcal K_2(t,y)&= e^{-t/2}
	\bigg(
	\frac{t^2}{16 \omega^2} I_2 (\frac \omega 2)
	- \bigg( \frac{t}{4 \omega} + \frac{y^2}{4 \omega ^3} \bigg) I_1(\frac \omega 2)
	+ \bigg( \frac 1 8 + \frac{t^2}{16 \omega^2} \bigg) I_0 (\frac \omega 2)
	\bigg),
	\end{align*}
$I_1$ and $I_2$ are first kind modified Bessel functions of first and second order,
and
	\[
	\widetilde{\mathcal K_j}(t,y)
	= \begin{cases}
	\mathcal K_j(t,y) & \mathrm{if} \quad |y| \leq t,\\
	0 & \mathrm{if} \quad |y| \geq t
	\end{cases}
	\]
for $j=1,2$.
For the detail of $\mathcal K_1$ and $\mathcal K_2$,
we refer the reader \cite{FG04o} and reference therein,
for example.
%]]]

%[[[ \label{Lemma:LinearEstimate1}
Then we have the following estimates:
\begin{Lemma}
\label{Lemma:LinearEstimate1}
Let $\phi \in L^1 \cap L^\infty (\mathbb R; [0,\infty)) $
satisfy \eqref{eq:IC2} -- \eqref{eq:IC5}.
Then $\partial_t S(t) \phi$ enjoys the pointwise estimate
	\[
	| \partial_t S(t) \phi(x)|
	\leq C_{\sigma} t^{-2(1-\sigma)} S(t) \phi(x)
	\]
for any $t \geq 2$, $x \in \mathbb R$, and $\sigma \in (1/2,1)$,
where $C_{\sigma}$ is a positive constant depending on $\sigma$
and $\phi$.
\end{Lemma}

\begin{proof}
We first note that
Lemma \ref{Lemma:0-mBessel} and the assumption \eqref{eq:IC2} imply
that the following estimates hold:
	\begin{align}
	S(t) \phi(x)
	\geq \int_{-1}^1 e^{\frac{\omega-t}{2}}
	\frac{\phi(x-y)}{\sqrt{\langle \omega \rangle}} dy
	\geq C t^{-\frac{1}{2}} \phi(x).
	\label{eq:Slower}
	\end{align}

We claim that
the estimate
	\[
	e^{-t/2} \frac{\phi(x+t)+\phi(x-t)}{2}
	\leq C e^{-t/4} \phi(x)
	\]
holds.
Indeed, if $|x| \leq 2 t$, then
the estimates
	\[
	e^{-t/2} \frac{\phi(x+t)+\phi(x-t)}{2}
	\leq C e^{-t/2}
	\leq C e^{-t/4} \phi(t)
	\leq C e^{-t/4} \phi(x).
	\]
follows from \eqref{eq:IC5} and \eqref{eq:IC3}.
Moreover, if $|x| \geq 2 t$, \eqref{eq:IC3} implies that
	\[
	e^{-t/2} \frac{\phi(x+t)+\phi(x-t)}{2}
	\leq e^{-t/2} \phi(x).
	\]
This and \eqref{eq:Slower} imply the assertion
for the first term in \eqref{eq:St}.

Now we show
	\[
	\int_{-t}^t \mathcal K_1(t,y) \phi(x-y) dy
	\leq C t^{-\frac{1}{2}-2(1-\sigma)} \phi(x).
	\]
We separate the integral domain into the cases:
	\begin{align*}
	D_1 &= \{ y \in [-t,t] \mid \ \omega \leq 1 \},\\
	D_2 &= \{ y \in [-t,t] \mid \ \omega \geq 1,\ |y| \leq t^\sigma \},\\
	D_3 &= \{ y \in [-t,t] \mid \ \omega \geq 1,\ |y| \geq t^\sigma \}.
	\end{align*}

On $D_2$,
we use the estimate
	\begin{align}
	\mathcal K_1(t,y)
	\leq C \omega^{-3/2} e^{\frac{\omega-t}{2}}
	+ C \bigg( \frac{t}{\omega}-1 \bigg) \omega^{-1/2} e^{\frac{\omega-t}{2}}.
	\label{eq:K1}
	\end{align}
For the detail of \eqref{eq:K1},
we refer \cite[Proof of Proposition 2.1]{MN03} for instance.
We note that the estimate
	\[
	\frac{t}{\omega}-1
	= \frac{y^2}{\omega(t+\omega)}
	\leq t^{-2(1-\sigma)}.
	\]
holds on $D_2$ because $\omega \sim t$ on $D_2$.
Then the estimate above, \eqref{eq:K1}, and \eqref{eq:Slower} imply that
the following estimates hold on $D_2$:
	\[
	\int_{D_2} \mathcal K_1(t,y) \phi(x-y) dy
	\leq C t^{-2(1-\sigma)}
	\int_{D_2} \frac{e^{(\omega-t)/2}}{\omega^{1/2}} \phi(x-y) dy
	\leq C t^{-2(1-\sigma)} S(t) \phi(x).
	\]

To estimate the integrals on $D_1$ and $D_3$,
we separate the cases where $|x| \leq 2 t$ and $|x| \geq 2 t$.

Assume $|x| \leq 2 t$.
The estimate $\mathcal K_1(t,y) \leq C e^{-t/2}$
holds on $D_1$ from the analyticity of $\mathcal K_1(t,y)$
with respect to $\omega$.
Then \eqref{eq:IC5} and \eqref{eq:IC3} again imply the estimate
	\[
	\int_{D_1} \mathcal K_1(t,y) \phi(x-y) dy
	\leq C t e^{-t/2}
	\leq C t^{-\frac{1}{2}-2(1-\sigma)} \phi(x).
	\]
On $D_3$, the estimate $e^{\frac{\omega-t}{2}} \leq C e^{-t^{2\sigma-1}/4}$
holds.
Therefore, we have
	\[
	\int_{D_3} \mathcal K_1(t,y) \phi(x-y) dy
	\leq C e^{-t^{2\sigma-1}/4}
	\]
and this implies the assertion on $D_3$.

In the case where $|x| \geq 2 t$,
we estimate the integral on $D_1$ and $D_3$ similarly
by using the estimate
	\[
	\phi(x-y) \leq C \phi(x)
	\]
following from $|y| \leq t$ and \eqref{eq:IC3}.
\end{proof}
%]]]

%[[[ \label{Lemma:LinearEstimate2}
\begin{Lemma}
\label{Lemma:LinearEstimate2}
Let $\phi $ satisfy the in Theorem \ref{Theorem:Main}.
Then $\partial_t^2 S(t) \phi$ enjoys the pointwise estimate
	\[
	| \partial_t^2 S(t) \phi(x)|
	\leq C_{\sigma} t^{-4(1-\sigma)} S(t) \phi(x)
	\]
for any $t \geq 2$, $x \in \mathbb R$, and $\sigma \in (1/2,1)$,
where $C_{\sigma}$ is a positive constant depending on $\sigma$.
\end{Lemma}

\begin{proof}
We note that the estimate
	\[
	e^{-t/2} \bigg| \frac{\dot \phi (x+t)- \dot \phi(x-t)}{2} \bigg|
	+ e^{-t/2} \bigg( \frac{t}{16} - \frac 1 2 \bigg)
	(\phi(x+t)+\phi(x-t))
	\leq C e^{-t/4} \phi(x).
	\]
is shown similarly to the proof of Lemma \ref{Lemma:LinearEstimate1}
thanks to \eqref{eq:IC4}.

Next, we separate the integral domain into three cases
as in the proof of Lemma \ref{Lemma:LinearEstimate1}.
The estimate
	\[
	\int_{D_1 \cup D_3} \mathcal K_2(t,y) \phi(x-y) dy
	\leq C t^{-\frac{1}{2}-4(1-\sigma)} \phi(x).
	\]
follows similarly to the proof of Lemma \ref{Lemma:LinearEstimate1}.
In order to treat the integral on $D_2$,
we use the following estimate:
	\begin{align}
	\mathcal K_2(t,y)
	\leq C \bigg( \frac{1}{\omega^{2}} + \frac{y^2}{\omega^{3}}
	+ \frac{y^4}{\omega^{2}(t+\omega)^2} \bigg)
	\frac{e^{\frac{\omega-t}{2}}}{\sqrt \omega}
	\label{eq:K4}
	\end{align}
holding fro $\omega \geq 1$.
We note that \eqref{eq:K4} is similarly shown to \eqref{eq:K1}.
On $D_2$, we further estimate
	\[
	\bigg(
		\frac{1}{\omega^{2}}
		+ \frac{y^2}{\omega^{3}}
		+ \frac{y^4}{\omega^{2}(t+\omega)^2}
	\bigg)
	\leq C ( t^{-2} + t^{-3+2\sigma} + t^{-4(1-\sigma)} ).
	\]
Then the estimate above, \eqref{eq:K4}, and \eqref{eq:Slower} imply that
	\[
	\int_{D_2} \mathcal K_2(t,y) \phi(x-y) dy
	\leq C t^{-4(1-\sigma)} S(t) \phi(x)
	\]
because the relations $-4+4\sigma > -3+2\sigma$ and $-4+4\sigma > -2$
are equivalent to $\sigma > 1/2$.
Therefore, the assertion holds.
\end{proof}
%]]]

%]]]

%[[[ \section{Proof of Main Statemanes}
\section{Proof of Main Statements}
\label{Section:Proof}
%[[[ First we give the definition of the supersolution to \eqref{eq:DW};
First, we give the definition of the supersolution to \eqref{eq:DW};
\begin{Definition}
A function $H^\ast \in C^2((0,\infty) \times \mathbb R) \cap C([0,\infty);L^1 \cap L^\infty)$  is called supersolution if
	\begin{align*}
	\begin{cases}
	\partial_t^2 H^\ast - \Delta H^\ast + \partial_t H^\ast + (H^\ast)^p \geq 0,
	&(t,x) \in (0,\infty) \times \mathbb R,\\
	H^\ast(0,x) \geq 0,
	&x \in \mathbb R,\\
	\partial_t H^\ast(0,x) + \frac 1 2 H^\ast(0,x) \geq 0,
	&x \in \mathbb R.
	\end{cases}
	\end{align*}
\end{Definition}
Indeed,
if such $H^\ast$ exists,
under initial condition
	\[
	0\leq u_0 \leq H^\ast(0),
	\quad 0 \leq u_1 + \frac 1 2 u_0
	\leq \partial_t H^\ast(0) + \frac 1 2 H^\ast(0),
	\]
the comparison principle holds:
the solution $u$ to \eqref{eq:DW} enjoys the pointwise estimate
	\[
	0 \leq u(t,x) \leq H^\ast(t,x)
	\]
for almost every $t>0$ and $x \in \mathbb R$.
For details, see Proposition \ref{Proposition:Comparison} in Appendix.
%]]]

%[[[ \subsection{Proof of Theorem \ref{Theorem:Main}}
\subsection{Proof of Theorem \ref{Theorem:Main}}
%[[[ In this subsection,
In this subsection,
we construct a supersolution to \eqref{eq:DW}
and show the decay estimate of the solution $u$.
%]]]

%[[[ In order to construct a supersolution,
In order to construct a supersolution,
we use \eqref{eq.defwg1} and replacing $\varepsilon$
by $f \in \mathbb{R}$ in \eqref{eq.defwg1} we put
	\[
	g(t,f)
	= \bigg( \frac{p (p+1) f^{p-1}}{(p-1)^2 f^{p-1} t + p^2+p} \bigg)^{1/(p-1)}
	\]
and $H(t,f) = w(t) g(t,f)$.
Further we define the pull back of $H$ with $f = u_L(t,x),$ i.e.
	\[
	H^\ast(t,x) = H(t, u_L(t,x))
	\]
where
$
u_L(t,x)
= \varepsilon S(t) \phi(x)
+ \varepsilon \partial_t S(t) \phi(x)
$.
We note that $u_L(t,x) >0$ holds for any $t>0$ and $x \in \mathbb R$.
We shall show that $H^\ast(t,x)$ is a supersolution
for $t \geq T_0$ with sufficiently large $T_0$ and with $\varepsilon$ small.
%]]]

%[[[ Direct computations imply the following identities,
Direct computations imply the following identities,
\begin{Lemma}
	We have
	\begin{equation}
	\begin{aligned}
	& H^\ast = w g^*, \ \ \  \partial_t (H^\ast)
	 = (\partial_t H)^* + (\partial_f H)^* \partial_t u_L,	  \\
	 & \partial_t^2 H^\ast
	=(\partial^2_t H)^*+ 2(\partial_t\partial_f H)^* \partial_t u_L + (\partial_f^2 H)^* (\partial_t u_L)^2 + (\partial_f H)^* (\partial_t^2 u_L) \\
	& \partial_x H^\ast
	 = (\partial_f H)^* \partial_x u_L,\\
	&\Delta H^\ast =
 (\partial^2_f H)^*  (\partial_x u_L)^2 + (\partial_f H)^* \Delta u_L .
	\end{aligned}
\end{equation}
\end{Lemma}

\begin{proof}
\begin{equation}
	\begin{aligned}
	& H^\ast = w g^*, \ \ \  \partial_t (H^\ast)
	= g^* \partial_t w + w (\partial_t g)^*
	+ w (\partial_f g)^* \partial_t u_L = (\partial_t H)^* + (\partial_f H)^* \partial_t u_L,
	\end{aligned}
\end{equation}
 Moreover,
 \begin{equation}
	 \begin{aligned}
		 & \partial_t^2 H^\ast
	= g^* \partial_t^2 w
	+ 2 \partial_t w (\partial_t g)^*
	+ w (\partial_t^2 g)^*
	+ 2 \partial_t w (\partial_f g)^* \partial_t u_L \\
	&
	+ 2 w (\partial_t \partial_f g)^* \partial_t u_L
	+ w (\partial_f^2 g)^* (\partial_t u_L)^2
	+ w (\partial_f g )^*\partial_t^2 u_L\\
	& =(\partial^2_t H)^* +  2 \partial_t w (\partial_f g)^* \partial_t u_L
	+ 2 w (\partial_t \partial_f g)^* \partial_t u_L
	+ w (\partial_f^2 g)^* (\partial_t u_L)^2
	+ w (\partial_f g )^*\partial_t^2 u_L  \\
	&= (\partial^2_t H)^*+ 2(\partial_t\partial_f H)^* \partial_t u_L + (\partial_f^2 H)^* (\partial_t u_L)^2 + (\partial_f H)^* (\partial_t^2 u_L)
	 \end{aligned}
 \end{equation}
	and
	\begin{equation}
		\begin{aligned}
		  & \partial_x H^\ast
	= w (\partial_f g)^* \partial_x u_L = (\partial_f H)^* \partial_x U_L,\\
	&\Delta H^\ast
	= w (\partial_f^2 g)^* (\partial_x u_L)^2
	+ w ( \partial_f g)^* \Delta u_L \\
	&= (\partial^2_f H)^*  (\partial_x u_L)^2 + (\partial_f H)^* \Delta u_L.
		\end{aligned}
	\end{equation}
	This completes the proof.
\end{proof}
%]]]

%[[[ Further we have
Further we have
\begin{equation}
\begin{aligned}
	&\partial_t^2 H^\ast +  \partial_t H^\ast - \Delta H^\ast + (H^\ast)^p \\
	&= ((\partial_t^2 H)^* + (\partial_t H)^* + (H^p)^*) + (\partial_f H)^*(\partial_t^2 u_L + \partial_t u_L - \Delta u_L) \\
  & + 2 (\partial_t\partial_f H)^* \partial_t u_L + (\partial^2_f H)^* \left( (\partial_t u_L)^2- (\partial_x u_L)^2 \right)
\end{aligned}
\end{equation}
Using the fact that $(\partial_t^2 + \partial_t - \Delta) u_L = 0$, we get

  \begin{equation}
\begin{aligned}
	&\partial_t^2 H^\ast +  \partial_t H^\ast - \Delta H^\ast + (H^\ast)^p \\
	&= ((\partial_t^2 H)^* + (\partial_t H)^* + (H^p)^*)  \\
  & + 2 (\partial_t\partial_f H)^* \partial_t u_L + (\partial^2_f H)^* \left( (\partial_t u_L)^2- (\partial_x u_L)^2 \right)
\end{aligned}
\end{equation}
Using the estimate \eqref{eq.odest6} of
Proposition \ref{Proposition:ODESuperSolution} we can write
	\begin{align}
	&\partial_t^2 H^\ast + \partial_t H^\ast - \Delta H^\ast + (H^\ast)^p
	\nonumber\\
	&\geq  \frac{C_{p,\alpha}}{(2+t)^{1+\alpha}} H^\ast
	+ \frac{C_{p,\alpha}}{p+1} (g^*)^{p-1} H^\ast
	+ 2 (\partial_t \partial_f H)^* \partial_t u_L
	+ (\partial_f^2 H )^* ( ( \partial_t u_L)^2 - (\partial_x u_L)^2)
	\nonumber\\
	& \geq \frac{C_{p,\alpha}}{(2+t)^{1+\alpha}} H^\ast
	+ \frac{C_{p,\alpha}}{p+1} (g^*)^{p-1} H^\ast
	+  2 (\partial_t \partial_f H)^* \partial_t u_L
	+ (\partial_f^2 H )^* ( \partial_t u_L)^2
	\label{eq.lbeq2}
	\end{align}
due to Corollary \ref{cor1}.
Further, we quote Corollary \ref{cor1} and estimate the terms
	\[
	2 (\partial_t \partial_f H)^* \partial_t u_L
	, (\partial_f^2 H )^* ( \partial_t u_L)^2
	\]
from below.
Indeed, from Corollary \ref{cor1}, we see that
$H \leq u_L$, $g\leq u_L$ so we get
\begin{align}
	& \left| (\partial_t\partial_f H)^* \partial_t u_L \right|
	\nonumber\\
	& =
	\bigg| \bigg( \frac{H^\ast}{u_L} \bigg)^{p-1}
	\frac{\alpha}{(p-1)} \frac{1}{(2^{1/\alpha}+t)^{\alpha+1}} H^\ast
	- \bigg( \frac{g^*}{u_L} \bigg)^{p-1}
	\frac{p-1}{(p+1)} (g^*)^{p-1} H^\ast \bigg| \left|\frac{\partial_t u_L}{u_L} \right|
	\nonumber\\
	& \leq C \left(\frac{H^\ast}{\langle t \rangle^{\alpha+1}}
	+ |g^*|^{p-1} H^\ast  \right) \left|\frac{\partial_t u_L}{u_L} \right|
	\nonumber\\
	& \leq C \left( \frac{H^\ast}{\langle t \rangle^{\alpha+1}}
	+ (g^*)^{p-1} H^\ast\right)  \frac{1}{\langle t \rangle^{2(1-\sigma)}}
	\leq C \frac{H^\ast}{\langle t \rangle^{2(1-\sigma)+\alpha+1}}
	+ C (g^*)^{p-1} H^\ast  \frac{1}{\langle t \rangle^{2(1-\sigma)}} .
	\end{align}
In a similar way, we find
	\begin{align*}
	\left|(\partial_f^2 H)^* ( \partial_t u_L)^2 \right|
	&=
	\left| p w (g^*)^{2p-1} u_L^{-2p} - p w (g^*)^{p} u_L^{-p-1} \right|
	( \partial_t u_L)^2 \\
	&\leq C \left| H^\ast (g^*)^{p-1} u_L^{-p-1} \right| ( \partial_t u_L)^2\\
	&= C \bigg( \frac{g^\ast}{u_L} \bigg)^{p-1} H^\ast
	\left(\frac{ \partial_t u_L}{u_L} \right)^2
	\leq C H^\ast\langle t \rangle^{-4(1-\sigma)}.
	\end{align*}
Hence,
	\begin{align}
	\left| (\partial_t\partial_f H)^* \partial_t u_L \right|+ \left|(\partial_f^2 H)^*  ( \partial_t u_L)^2 \right|
	\leq \frac{C H^\ast}{\langle t \rangle^{2(1-\sigma)+1+\alpha}}
	+ \frac{C H^\ast}{\langle t \rangle^{4(1-\sigma)}}
	+ C (g^*)^{p-1} H^\ast \frac{1}{\langle t \rangle^{2(1-\sigma)}},
	\end{align}
and we can continue \eqref{eq.lbeq2} as follows
	\begin{align}
	&\partial_t^2 H^\ast + \partial_t H^\ast - \Delta H^\ast + (H^\ast)^p
	\nonumber\\
	& \geq \frac{C_{p,\alpha}}{(2+t)^{1+\alpha}} H^\ast
	- \frac{C H^\ast}{\langle t \rangle^{2(1-\sigma)+\alpha+1}}
	- \frac{C H^\ast}{\langle t \rangle^{4(1-\sigma)}}
	\nonumber\\
	&+ C_{p,\alpha} (g^*)^{p-1} H^\ast
	- C (g^*)^{p-1} H^\ast \frac{1}{\langle t \rangle^{2(1-\sigma)}}
	\label{eq.lbeq43}
	\end{align}
The positiveness of the term
	\begin{equation}
	\frac{C_{p,\alpha} H^\ast}{(2+t)^{1+\alpha}}
	- \frac{C H^\ast}{\langle t \rangle^{2(1-\sigma)+\alpha+1}}
	- \frac{C H^\ast}{\langle t \rangle^{4(1-\sigma)}}
	\end{equation}
for large $t$ is guaranteed by
	\[
	1 + \alpha < 4(1-\sigma)
	\Leftrightarrow 3 - 4 \sigma > \alpha.
	\]
We have also
	\[
	C_{p,\alpha} (g^*)^{p-1} H^\ast
	- C (g^*)^{p-1} H^\ast \frac{1}{\langle t \rangle^{2(1-\sigma)}}>0
	\]
for large $t.$
Therefore, choosing first $\sigma \in (1/2,3/4),$ and then	$\alpha \in (0,3-4\sigma)$ we see  that
there exists $T_0 \geq 0$ with
	\[
	\partial_t^2 H^\ast(t,x) + \partial_t H^\ast(t,x) - \Delta H^\ast(t,x) + (H^\ast(t,x))^p
	> 0
	\]
for $t \geq T_0$.
%]]]

%[[[ We note that
We note that
	\[
	H(T_0,x)
	\geq \frac{1}{T_0^{p-1}} u_L(T_0,x)
	\geq \varepsilon \frac{1}{T_0^{p-1/2}} \phi(x).
	\]
%]]]

%[[[ Moreover, we compute
Moreover, we compute
	\[
	\partial_t H(T_0) + \frac{1}{2} H(T_0)
	\geq - \frac{p-1}{p(p+1)} g(T_0)^{p-1} H(T_0)
	+ \frac{g(T_0)^{p-1}}{u_L(T_0)^{p-1}} H(T_0) \frac{\partial_t u_L(T_0)}{u_L(T_0)}
	+ \frac 1 2 H(T_0).
	\]
Since we assume $\varepsilon$ small,
we have $\partial_t H(T_0) + H(T_0)/2 \geq 0$.
%]]]

%[[[ \begin{Remark}
\begin{Remark}
In the argument above,
it is crucial to cancel $(\partial_f^2 H)^\ast (\partial_x u_L)^2$
by its sign.
Indeed, a similar argument of Corollary \ref{cor1} may imply that
	\[
	(\partial_f^2 H)^\ast (\partial_x u_L)^2
	\leq \frac{C H^\ast}{\langle t \rangle^{1-\delta}}
	\]
with some small number $\delta$.
This implies that
our argument is insufficient to control
$(\partial_f^2 H)^\ast (\partial_x u_L)^2$
by $C H^\ast \langle t \rangle^{-1-\alpha}$.
We also note that
our approach is insufficient to construct subsolutions in a similar manner
because $(\partial_f^2 H)^\ast (\partial_x u_L)^2$
has to be treated in this case.
On the other hand,
this is not a technical problem but a natural fact
because for $\rho > \frac{2}{p-1}$,
the RHS of the second estimate of \eqref{eq:Main}
does not correspond to the estimate \eqref{eq:ideal}.
\end{Remark}
%]]]
%]]]

%[[[ \subsection{Proof of Corollary \ref{Corollary:Main}}
\subsection{Proof of Corollary \ref{Corollary:Main}}
Corollary \ref{Corollary:Main} follows from the pointwise estimate
	\[
	S(t) \phi
	\leq C e^{t \Delta} \phi
	\]
holding for integrable non-negative function $\phi$.
Therefore, we have
	\[
	H^\ast(t,x) \leq C G^\ast(t,x)
	\]
for any $t$ and $x$.
This and Proposition \ref{Proposition:HeatSuperSolution} imply the assertion.
%]]]
%]]]

%[[[ \section*{Appendix}
\section*{Appendix}
% label is changed to A.1
\refstepcounter{section}
\renewcommand{\thesection}{A}
%[[[ \label{Proposition:HeatSuperSolution}
\begin{Proposition}
\label{Proposition:HeatSuperSolution}
Let $1 < p < 3$ and $1 < \rho < 2/(p-1)$.
Then the following estimate holds as $t \to \infty$:
	\[
	\Big\|
	\big((p-1)t+(e^{t \Delta} \langle \cdot \rangle^{-\rho})^{1-p} \big)^{-1/(p-1)}
	\Big\|_{L^q}
	= O( t^{\frac{1}{q \rho (p-1)} -\frac{1}{p-1}}).
	\]
\end{Proposition}

\begin{proof}
We note that the estimates
	\begin{align*}
	e^{t \Delta} \langle \cdot \rangle^{-\rho}(x)
	&\leq C \begin{cases}
	1 & \mathrm{if} \quad |x| \leq t^{\tfrac{1}{\rho(p-1)}},\\
	|x|^{-\rho} & \mathrm{if} \quad |x| \geq t^{\tfrac{1}{\rho(p-1)}}
	\end{cases},\\
	e^{t \Delta} \langle \cdot \rangle^{-\rho}(x)
	&\geq C \begin{cases}
	0 & \mathrm{if} \quad |x| \leq t^{\tfrac{1}{\rho(p-1)}},\\
	|x|^{-\rho} & \mathrm{if} \quad |x| \geq t^{\tfrac{1}{\rho(p-1)}}
	\end{cases}
	\end{align*}
follow from the assumption $1 < \rho < 2/(p-1)$.
Indeed, when $|x| \geq t^{\tfrac{1}{\rho(p-1)}}$,
we estimate the upper bound as follows:
	\begin{align*}
	e^{t \Delta} \langle \cdot \rangle^{-\rho}(x)
	&= \frac{1}{\sqrt{4\pi t}}
	\int_{|y| \leq |x|/2} e^{-y^2/4t} \langle x - y \rangle^{-\rho} dy
	+ \frac{1}{\sqrt{4\pi t}}
	\int_{|y| \geq |x|/2} e^{-y^2/4t} \langle x - y \rangle^{-\rho} dy\\
	&\leq C \langle x \rangle^{-\rho}
	+ C e^{-|x|^{2-\rho(p-1)}} |x|^{-\frac{1}{2 \rho (p-1)}}
	\int_{|y| \geq |x|/2} \langle y \rangle^{-\rho} dy
	\leq C \langle x \rangle^{-\rho}.
	\end{align*}
Then we estimate
	\begin{align*}
	\int \big((p-1)t+(e^{t \Delta} \langle \cdot \rangle^{-\rho})^{1-p} \big)^{-q/(p-1)} dx
	&\leq C \int_{|x| \leq t^{\frac{1}{\rho(p-1)}}} t^{-\tfrac{q}{p-1}} dx
	+ C \int_{|x| \geq t^{\frac{1}{\rho(p-1)}}} \langle x \rangle^{-q \rho} dx\\
	&\leq C t^{\frac{1}{\rho(p-1)} - \frac{q}{p-1}}.
	\end{align*}
The lower bound is estimated similarly.
\end{proof}
%]]]

%[[[ \label{Proposition:HeatSuperSolution2}
\begin{Proposition}
\label{Proposition:HeatSuperSolution2}
Let $1 < p < 3$ and $\rho \geq 2/(p-1)$.
Then the following estimate holds for sufficiently large $t$:
	\[
	\Big\|
	\big((p-1)t+(e^{t \Delta} \langle \cdot \rangle^{-\rho})^{1-p} \big)^{-1/(p-1)}
	\Big\|_{L^q}
	\geq C t^{\frac{1}{2 q} -\frac{1}{p-1}} \sqrt{\log(t)}.
	\]
\end{Proposition}

\begin{proof}
We note that for $x \geq 2$, we estimate
	\[
	e^{t \Delta} \langle \cdot \rangle^{-\rho}(x)
	\geq \frac{1}{\sqrt t} e^{-x^2/2t}.
	\]
We also have
	\[
	\bigg( \frac{1}{\sqrt t} e^{-x^2/2t} \bigg)^{1-p}
	\leq t
	\]
for $|x| \leq \sqrt{\frac{2t}{p-1} \log(t^{\frac{3-p}{2}})} =: r(t)$.
Therefore, we estimate
	\begin{align*}
	\Big\|
	\big((p-1)t+(e^{t \Delta} \langle \cdot \rangle^{-\rho})^{1-p} \big)^{-1/(p-1)}
	\Big\|_{L^q}
	\geq C
	\Big\|
	t^{-1/(p-1)}
	\Big\|_{L^q( 2 \leq x \leq r(t))}
	\geq C t^{\frac{1}{2q}-\frac{1}{p-1}} \sqrt{\log(t)}.
	\end{align*}
\end{proof}
%]]]

%[[[ \label{Proposition:Comparison}
\begin{Proposition}
\label{Proposition:Comparison}
Let $\overline u, \underline u
\in C([0,\infty); L^1 \cap L^p(\mathbb R;[0,\infty)))$
be mild solutions to the following inequality:
	\[
	\begin{cases}
	\partial_t^2 \overline u + \partial_t \overline u - \Delta \overline u
	+ \overline u^p \geq 0,
	&(t,x) \in (0,\infty) \times \mathbb R\\
	\partial_t \underline u + \partial_t \underline u - \Delta \underline u
	+ \underline u^p \leq 0,
	\end{cases}
	\]
under the initial conditions
	\[
	\overline u(0,x) \geq \underline u(0,x),
	\quad
	\partial_t \overline u(0,x) + \frac 1 2 \overline u(0,x)
	\geq \partial_t \underline u(0,x) + \frac 1 2 \underline u(0,x).
	\]
We further assume that we have
	\begin{align}
	\sup_{\tau,y} \Big(\overline u(\tau,y) + \underline u(\tau,y) \Big)^{p-1}
	\leq \frac{1}{4p}.
	\label{eq:sup}
	\end{align}
Then the following inequality holds:
	\[
	\overline u(t,x) \geq \underline u(t,x)
	\]
for any $t > 0$ and $x \in \mathbb R$.
\end{Proposition}
\begin{proof}
Let $\overline w = e^{t/2} \overline u$
and $\underline w = e^{t/2} \underline u$.
Then we have
	\begin{align*}
	\partial_t^2 \overline w - \Delta \overline w
	&\geq \frac 1 4 \overline w + \overline{u}^{p-1} \overline w,\\
	\partial_t^2 \underline w - \Delta \underline w
	&\leq \frac 1 4 \underline w + \underline{u}^{p-1} \underline w.
	\end{align*}
The inqualities above imply that
by putting $W = \overline w - \underline w$,
the following inequality holds:
	\[
	W(t,x)
	\geq \frac{1}{2} \int_0^t \int_{\tau -t}^{t-\tau}
	\frac 1 4 W(\tau, x+y)
	+ R(\tau, x+y) dy d\tau,
	\]
where
	\[
	R(\tau,y)
	= e^{t/2}
	\Big( \overline{u}^{p}(\tau, y) - \underline{u}^{p}(\tau, y) \Big).
	\]
The estimate
	\[
	|R(\tau,y)|
	\leq p \Big(\overline u(\tau,y) + \underline u(\tau,y) \Big)^{p-1} |W(\tau,y)|
	\]
follows from a direct computation.
Now we put $W_-(\tau,y) = \max\{0,W(\tau,y)\}$.
Then,
the assumption \eqref{eq:sup} and the discussion above implies that
	\begin{align*}
	\inf_{x \in \mathbb R} W_-(t,x)
	&\geq \frac{1}{2} \int_0^t \int_{\tau -t}^{t-\tau}
	\frac 1 4 \inf_{x \in \mathbb R} W_-(\tau, y) dy d\tau\\
	&\geq t \int_0^t \inf_{x \in \mathbb R} W_-(\tau, y) d\tau.
	\end{align*}
Therefore, we see that $W_-(t,x) \geq 0$ holds for any $t>0$ and $x \in \mathbb R$.
\end{proof}
%]]]

%]]]

%[[[ \section*{Acknowldgement}
\section*{Acknowledgment}
% The authors are grateful to the referees
% for their careful reading and helpful comments.

The first author was supported in part by
JSPS Grant-in-Aid for Early-Career Scientists No. 24K16957 and 24H00024.
The second author was partially supported by Gruppo Nazionale per l'Analisi Matematica, project “Modelli nonlineari in presenza di interazioni puntuali” CUP - E53C22001930001, by the project PRIN  2020XB3EFL with the Italian Ministry of Universities and Research, by Institute of Mathematics and Informatics, Bulgarian Academy of Sciences and  by Top Global University Project, Waseda University.
%]]]

% use L1DecODE.bib
\bibliographystyle{plain}
\bibliography{dwsupersol}

\end{document}